\documentclass{amsart}

\usepackage{amsmath,amssymb,amsfonts}
\usepackage{ifpdf}
\usepackage{enumerate}
\usepackage{ulem}
\usepackage{leftidx}

\ifpdf

\else

\fi

\newcommand{\arxiv}[1]{\href{http://arxiv.org/abs/#1}{\tt arXiv:\nolinkurl{#1}}}
\newcommand{\doi}[1]{\href{http://dx.doi.org/#1}{{\tt DOI:#1}}}

\newcommand{\googlebooks}[1]{(preview at \href{http://books.google.com/books?id=#1}{google books})}

\def\RCS$#1: #2 ${\expandafter\def\csname RCS#1\endcsname{#2}}
\RCS$Revision: 1913 $ \RCS$Date: 2010-07-12 09:25:35 -0700 (Mon, 12 Jul 2010) $

\theoremstyle{plain}
\newtheorem{prop}{Proposition}[section]

\newtheorem{lem}[prop]{Lemma}

\newtheorem*{cor*}{Corollary}

\newtheorem*{defn*}{Definition}             
\newtheorem*{rem*}{Remark}               
\newtheorem*{ex*}{Example}                

\newtheorem*{mainthm}{Theorem}
\numberwithin{equation}{section}

\newcounter{comment}
\newcommand{\noop}[1]{}

\def\clap#1{\hbox to 0pt{\hss#1\hss}}

\def\semicolon{;}
\def\applytolist#1{
    \expandafter\def\csname multi#1\endcsname##1{
        \def\multiack{##1}\ifx\multiack\semicolon
            \def\next{\relax}
        \else
            \csname #1\endcsname{##1}
            \def\next{\csname multi#1\endcsname}
        \fi
        \next}
    \csname multi#1\endcsname}

\def\calc#1{\expandafter\def\csname c#1\endcsname{{\mathcal #1}}}
\applytolist{calc}QWERTYUIOPLKJHGFDSAZXCVBNM;
\def\bbc#1{\expandafter\def\csname bb#1\endcsname{{\mathbb #1}}}
\applytolist{bbc}QWERTYUIOPLKJHGFDSAZXCVBNM;
\def\bfc#1{\expandafter\def\csname bf#1\endcsname{{\mathbf #1}}}
\applytolist{bfc}QWERTYUIOPLKJHGFDSAZXCVBNM;

\renewcommand{\imath}{\mathfrak{i}}
\renewcommand{\jmath}{\mathfrak{j}}

\newcommand{\iso}{\cong}

\makeatletter
\newcommand{\hashdef}[2]{\@namedef{#1}{#2}}
\newcommand{\hashlookup}[1]{\@nameuse{#1}}
\makeatother

\IfFileExists{../../graphs/lookup.tex}%
	{\newcommand{\pathtographs}{../../graphs/}}%
	{\newcommand{\pathtographs}{diagrams/graphs/}}

\hashdef{bwd1v1v1v1p1v0x1p0x1v1x0p0x1v0x1p1x0p1x0v0x1x0p0x1x0p1x0x1v0x1x0v1duals1v1v1x2v1x2v1x2x3v1}{C6B47C9F76B9B1C9}%
\hashdef{bwd1v1v1v1p1v1x0p0x1v1x0p1x0p0x1p0x1v1x0x0x0p0x1x0x0p0x0x0x1v1x0x0p1x0x0p0x1x0p0x1x1p0x0x1v0x0x1x0x0v1duals1v1v1x2v4x2x3x1v2x1x3x4x5v1}{08B56FAF803E9257}%

\newcommand{\bigraph}[1]{{\hspace{-3pt}\begin{array}{c}%
  \raisebox{-2.5pt}{\includegraphics[height=7.5mm]{\pathtographs \hashlookup{#1}}}%
\end{array}\hspace{-3pt}}}

\ifpdf
\usepackage[pdftex,plainpages=false,hypertexnames=false,pdfpagelabels]{hyperref}
\else
\usepackage[dvips,plainpages=false,hypertexnames=false]{hyperref}
\fi

\ifpdf
\usepackage[pdftex]{graphicx}
\usepackage[pdftex]{rotating}
\else
\usepackage[dvips]{graphicx}
\usepackage[dvips]{rotating}
\fi

\usepackage{microtype}

\usepackage{tikz}
\usetikzlibrary{shapes}
\usetikzlibrary{backgrounds}

\usepackage{color}

\title[An obstruction to subfactor principal graphs]{An obstruction to subfactor principal graphs from the graph planar algebra embedding theorem}

\author{Scott~Morrison}
\address{Mathematical Sciences Institute, the Australian National University
}%
\email{scott.morrison@anu.edu.au}

\address{%
\rm URL: \tt \url{http://tqft/net/}}

\begin{document}

\begin{abstract}
We find a new obstruction to the principal graphs of subfactors. It shows that in a certain family of 3-supertransitive principal graphs, there must be a cycle by depth 6, with one exception, the principal graph of the Haagerup subfactor. 
\end{abstract}

\maketitle

A $II_1$ subfactor is an inclusion $A \subset B$ of infinite von Neumann algebras with trivial centre and a compatible trace with $\operatorname{tr}(1) = 1$. In this setting, one can analyze the bimodules $\otimes-$generated by ${}_A B_B$ and ${}_B B_A$. The principal graph of a subfactor has as vertices the simple bimodules appearing, and an edge between vertices $X$ and $Y$ for each copy of $Y$ appearing inside $X \otimes B$. It turns out that this principal graph is a very useful invariant of the subfactor (at least in the amenable case), and many useful combinatorial constraints on graphs arising in this way have been discovered. As examples, see \cite{1007.1730, math/1007.1158, 1207.5090, 1307.5890}. Moreover, with sufficiently powerful combinatorial (and number theoretic, c.f. \cite{MR2307421, MR2472028, 1004.0665, 0810.3242}) constraints in hand, it has proved possible to enumerate all possible principal graphs for subfactors with small index. This approach was pioneered by Haagerup in \cite{MR1317352}, and more recently continued, resulting in a classification of subfactors up to index 5 \cite{1007.1730,index5-part2,index5-part3, index5-part4, 1304.6141}. 

In this note we demonstrate the following theorem, providing a combinatorial constraint on the principal graph of a subfactor, of a rather different nature than previous results.
\begin{mainthm}
\label{main-theorem}
If the principal graph of a 3-supertransitive $II_1$ subfactor begins as 
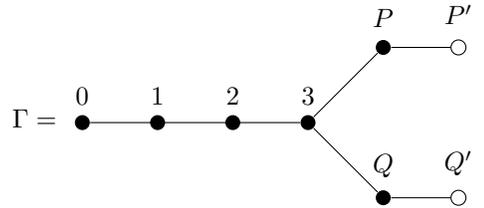
\begin{equation*}
\Gamma = 
\begin{tikzpicture}[baseline=-2pt, vertex/.style=fill,circle,inner sep=2,emptyvertex/.style=draw,circle,inner sep=2]
\node[vertex, label=above:$0$] (0) at (0,0) {}; 
\node[vertex, label=above:$1$] (1) at (1,0) {}; 
\node[vertex, label=above:$2$] (2) at (2,0) {}; 
\node[vertex, label=above:$3$] (3) at (3,0) {}; 
\node[vertex, label=above:$P$] (P) at (4,1) {}; 
\node[vertex, label=above:$Q$] (Q) at (4,-1) {}; 
\node[emptyvertex, label=above:$P'$] (P') at (5,1) {}; 
\node[emptyvertex, label=above:$Q'$] (Q') at (5,-1) {}; 
\draw (0) -- (3) -- (P) -- (P') (3) -- (Q) -- (Q');
\end{tikzpicture}
,
\end{equation*}
and there is no vertex at depth $6$ which is connected to both $P'$ and $Q'$, then the subfactor must have the same standard invariant as the Haagerup subfactor.
\end{mainthm}

(A subfactor being 3-supertransitive merely means that the principal graph begins with a chain of 3 edges. The hypothesis specifying the graph up to depth 5 is equivalent to the subfactor being 3-supertransitive and having `annular multiplicities' 10, as described in \cite{math/1007.1158}.)

This result uses the graph planar algebra embedding theorem, proved in \cite{MR2812459} and alternatively in the forthcoming \cite{tvc}. (The first result only holds in finite depth, while the second only assumes that the principal graph is locally finite; the result here inherits these restrictions. The principal graph of a finite index subfactor is always locally finite.)

This appears to be an instance of a potentially new class of obstructions to principal graphs of subfactors, somewhat different in nature from those derived by analyzing connections or quadratic tangles. It seems likely that many generalizations of this result are possible, especially if a more conceptual proof can be found.

As an example application, we can use this obstruction to rule out the existence of a subfactor with principal graph
$$\left(
\bigraph{bwd1v1v1v1p1v1x0p0x1v1x0p1x0p0x1p0x1v1x0x0x0p0x1x0x0p0x0x0x1v1x0x0p1x0x0p0x1x0p0x1x1p0x0x1v0x0x1x0x0v1duals1v1v1x2v4x2x3x1v2x1x3x4x5v1},
\bigraph{bwd1v1v1v1p1v0x1p0x1v1x0p0x1v0x1p1x0p1x0v0x1x0p0x1x0p1x0x1v0x1x0v1duals1v1v1x2v1x2v1x2x3v1}
\right)$$ at index $3+\sqrt{5}$. The possibility of such a subfactor arose during a combinatorial search for possible principal graphs, following on from \cite{1007.1730,index5-part2,index5-part3, index5-part4}. There is a bi-unitary connection on this principal graph, and this result immediately shows that the connection must not be flat.

Given a subfactor with principal graph beginning as $\Gamma$ in the theorem, in the corresponding planar algebra $P$ we have $P_{4,+} = TL_4 \oplus \bbC S$, where $S$ is a lowest weight vector (that is, in the kernel of each of the cap maps to $P_{3,\pm}$) and also a rotational eigenvalue with eigenvalue $\omega$ a fourth root of unity. In fact, $\omega = \pm 1$, as otherwise $P$ and $Q$ are dual to each other; in this case, the dual graph must begin the same way, and either Ocneanu's or Jones' triple point obstruction \cite{MR1317352,math/1007.1158,index5-part2} rules out a possible subfactor. The results of \cite{math/1007.1158} show that this element $S$, suitably normalized, satisfies the quadratic identity
\begin{equation}
\label{eq:S2}
S^2 = (1-r) S + r f^{(4)}
\end{equation}
where $r$ is the ratio $\dim P / \dim Q$ (or its reciprocal, if less than one) and $f^{(4)}$ is the 4-strand Jones-Wenzl idempotent. See \cite[Theorem 3.9]{0909.4099} for details. The main identity of \cite{math/1007.1158} shows that
\begin{equation*}
r = \begin{cases}
\frac{[5]+1}{[5]-1} & \text{when $\omega = +1$ and} \\
1 & \text{when $\omega = -1$.}
\end{cases}
\end{equation*}
Here $[5]$ denotes the quantum integer $q^{-4} + q^{-2} + 1 + q^2 + q^4$, where $q$ is a parameter determined by the (unknown) index of the subfactor $N \subset M$ by $[M: N] = [2]^2 = q^{-2} + 2 + q^2$.

The embedding theorems of \cite{MR2812459,tvc} show that there is a faithful map of planar algebras  $\varepsilon: P \to GPA(\Gamma)$. (See \cite{MR1865703, 0909.4099, 1208.3637, 1205.2742} for more details on the definition of the graph planar algebra, and examples of calculations therein.)
We thus consider the image $\varepsilon(S) \in GPA(\Gamma)_{4,+}$ in the graph planar algebra for $\Gamma$. The algebra $GPA(\Gamma)_{4,+}$ splits up as a direct sum of matrix algebras, corresponding to loops on $\Gamma$ that have their base-point and mid-point at specified vertices at even depths:
$$ GPA(\Gamma)_{4,+} \iso \bigoplus_{a,b} \cM_{a,b}.$$

For some, but not all, of these matrix algebras, every component corresponds to a loop which stays within the first five depths of the principal graph (that is, does not go above $P'$ and $Q'$ in the diagram above). In particular, these matrix algebras are all those $\cM_{a,b}$ where $a, b \in \{0,2,P,Q\}$, \emph{except} for $\cM_{P,P}$ and $\cM_{Q,Q}$. This condition only holds for $\cM_{P,Q}$ and $\cM_{Q,P}$ because of our additional hypothesis that $P'$ and $Q'$ are not both connected to some vertex at depth 6. We call the subalgebra of $GPA(\Gamma)_{4,+}$ comprising these matrix algebras $\cA$.

The condition that an element $S \in GPA(\Gamma)_{4,+}$ is a lowest weight rotational eigenvector with eigenvalue $\omega$ consists of a collection of linear equations $\cL_\omega$ relating the coefficients of various loops on $\Gamma$. (Notice that the coefficients in these equations depend on $q$, because the lowest weight condition depends on the dimensions of the objects in the principal graph, and these are determined by $q$ and $r$.) Some of these equations only involve loops supported in the first 5 depths of $\Gamma$, and we call these equations $\cL^+_\omega$. (Thinking of these equations as functionals on $GPA(\Gamma)_{4,+}$, we are taking the subset of functionals which are supported on the subalgebra $\cA$.) The solutions of $\cL_\omega$ are certainly a subspace of the solutions of $\cL^+_\omega$.

The general strategy is now straightforward.
\begin{itemize}
\item Identify the Jones-Wenzl idempotent $f^{(4)}$ in the matrix algebras $\cA$, as a function of the parameter $q$.
\item Solve the equations $\cL^+_\omega$, which must hold for any lowest weight rotational eigenvector with eigenvalue $\omega$, in the matrix algebras $\cA$, finding a linear subspace.
\item Analyze the equation $S^2 = (1-r) S + r f^{(4)}$ in this subspace.
\end{itemize}
This provides us with some quadratic equations in a small vector space over $\mathbb{Q}(q)$. In fact, by carefully choosing particular equations to consider, and by appropriate changes of variables, we can understand the solutions to these equations directly. We find, in \S \ref{sec:omega=+1}, that when $\omega = +1$, no solutions are possible. On the other hand we see in \S \ref{sec:omega=-1} that when $\omega = -1$ there is a discrete set of solutions if and only if a certain identity is satisfied by the parameter $q$:
$$q^8 - q^6 - q^4 - q^2 + 1 = 0.$$
The real solutions of this equation ensure that $q^2 + 2 + q^{-2} = (5+\sqrt{13})/2$, i.e. the index of our subfactor is exactly the index of the Haagerup subfactor. It is easy to see that the only principal graph extending the one we have specified up to depth 5 that has this index is the principal graph of the Haagerup subfactor. Moreover, previous classification results (c.f. \cite{MR1317352, MR1686551} and also \cite{MR2679382} for an alternative construction) show that there is a unique (up to complex conjugation) subfactor planar algebra at this index, besides the $A_\infty$ subfactor planar algebra.

The remainder of this paper consists of an analysis of the equations discussed above. Unfortunately, in the present state of development of computer algebra, solving them remains an art, not a science. Gr\"{o}bner basis algorithms, unsupervised, can not draw the necessary conclusions. Essentially what follows is an explanation of how to decide which equations to consider in which order, so that the solutions are at each step easy to extract. The derivation is, unfortunately, somewhat opaque, representing the human-readable form of an argument better carried out in conversation with a computer.

\section{Preliminary calculations}

We begin by noting the dimensions of all the vertices above, as functions of $q$ and $r$.
\begin{lem}
\begin{align*}
\dim(0) & = 1 \displaybreak[1] \\
\dim(1) & = q+q^{-1}  \displaybreak[1] \\
\dim(2) & = q^2 + 1 + q^{-2}  \displaybreak[1] \\
\dim(3) & = q^3 + q + q^{-1} + q^{-3}  \displaybreak[1] \\
\dim(P) & = \frac{r}{r+1}(q^4 + q^2 + 1 + q^{-2} + q^{-4})  \displaybreak[1] \\
\dim(Q) & = \frac{1}{r+1}(q^4 + q^2 + 1 + q^{-2} + q^{-4})  \displaybreak[1] \\
\dim(P') & = \frac{r}{r+1}(q^5 + 2 q^3+ 2q + 2q^{-1} + 2q^{-3} + q^{-5}) - (q^3 + q + q^{-1} + q^{-3}) \\
	      & = \frac{r}{r+1}(q^5 + q^{-5}) + \frac{r-1}{r+1}(q^3 + q + q^{-1} + q^{-3}) \\
\dim(Q') & = \frac{1}{r+1}(q^5 + 2 q^3+ 2q + 2q^{-1} + 2q^{-3} + q^{-5}) - (q^3 + q + q^{-1} + q^{-3})  \displaybreak[1] \\
	      & = \frac{1}{r+1}(q^5 + q^{-5}) - \frac{r-1}{r+1}(q^3 + q + q^{-1} + q^{-3})
\end{align*}
\end{lem}
\begin{proof}
This is just the condition that the dimensions form an eigenvector of the graph adjacency matrix, with eigenvalue $q+q^{-1}$, and the dimensions of the two vertices at depth 4 have ratio $r$.
\end{proof}

\begin{lem}
When $\omega = -1$, $r=1$, and 
\begin{align*}
\dim(P) = \dim(Q) & = \frac{1}{2}(q^4 + q^2 + 1 + q^{-2} + q^{-4}) \\
\dim(P')=\dim(Q') &= \frac{1}{2}(q^5 + q^{-5}).
\end{align*}
When $\omega = +1$, $$r= \frac{q^8 + q^6 + 2 q^4 + q^2 + 1}{q^8 + q^6 + q^2 + 1},$$
and
\begin{align*}
\dim(P) & = \frac{1}{2}(q^4+q^2+2+q^{-2}+q^{-4}) \\
\dim(Q) & = \frac{1}{2}(q^4+q^2+q^{-2}+q^{-4}) \\
\dim(P') & = \frac{1}{2}(q^5+q+q^{-1}+q^{-5}) \\
\dim(Q') & = \frac{1}{2}(q^5-q-q^{-1}+q^{-5})
\end{align*}
\end{lem}
\begin{proof}
The formulas for $r$ follows immediately from Theorem 5.1.11 from \cite{math/1007.1158}.
\end{proof}

\begin{lem}
Suppose $S$ is a  4-box. We will denote by $S(\cdots abc \cdots)$ the evaluation of $S$ on  certain loop; in each equation, the parts of the loops indicated by the ellipses are held fixed. If the loop passes through $b$ at either the base point or the mid point of the loop, we define $\kappa = 1$, and otherwise $\kappa = 1/2$.

If $S$ is a lowest weight 4-box,
\begin{align}
S(\cdots 010 \cdots) & = 0 \label{eq:lws-0} \\
S(\cdots 121 \cdots) & = - \left(\frac{1}{\dim(2)}\right)^\kappa S(\cdots 101 \cdots) \label{eq:lws-1} \\
S(\cdots 232 \cdots) & = -\left(\frac{\dim(1)}{\dim(3)}\right)^\kappa S(\cdots 212 \cdots) \label{eq:lws-2}  \\
S(\cdots P3P\cdots) & =  -\left(\frac{\dim(3)}{\dim(P')}\right)^\kappa S(\cdots P P' P \cdots) \label{eq:lws-P}, \\
S(\cdots Q3Q\cdots) & =  -\left(\frac{\dim(3)}{\dim(Q')}\right)^\kappa S(\cdots Q Q' Q \cdots) \label{eq:lws-Q}, \\
\intertext{and further} 
S(\cdots 3Q3 \cdots) & = -\left(\frac{\dim(Q)}{\dim(2)}\right)^\kappa S(\cdots 323 \cdots) -\left(\frac{1}{r}\right)^\kappa S(\cdots 3P3 \cdots) \label{eq:lws-3}
\end{align}
\end{lem}
\begin{proof}
These follow directly from the definition of lowest weight vector. Calculate the evaluations $\cap_i S(\cdots 1 \cdots)$, $\cap_i S(\cdots 2 \cdots)$, $\cap_i S(\cdots 3 \cdots)$ or $\cap_i S(\cdots P \cdots)$ as follows:
$$0 = \cap_i S(\cdots \alpha \cdots) = \sum_{\beta} \left(\frac{\dim(\beta)}{\dim(\alpha)}\right)^\kappa S(\cdots \alpha \beta \alpha \cdots).$$
Thus for example we have $0= {\dim(2)}^\kappa S(\cdots 121 \cdots) + \dim(0)^\kappa S(\cdots 101 \cdots)$.
\end{proof}

It is is easy to see that using Equations \eqref{eq:lws-1}, \eqref{eq:lws-2}, \eqref{eq:lws-P} and \eqref{eq:lws-Q} (and not needing Equation \eqref{eq:lws-3}) we can write the coefficient in a lowest weight 4-box of any loop supported in depths at most $5$ as a real multiple of the  coefficient of a corresponding `collapsed' loop which does not leave the immediate vicinity of the vertex $3$ (that is, supported on $2$, $3$, $P$ and $Q$). There are 81 such collapsed loops, and 24 orbits of the rotation group on the collapsed loops. Thus after fixing a rotational eigenvalue of $\omega=\pm 1$, we may work in a 24-dimensional space (different for each eigenvalue). Equation \eqref{eq:lws-0} ensures that any loop confined to the initial arm, and in particular the collapsed loop supported on $2$ and $3$, is zero. (An similar analysis of a lowest weight space in a graph planar algebra is described in more detail in \cite[\S A]{1208.3637}.) There remains all the instances of Equation \eqref{eq:lws-3}, which in fact cut down the space to either 4-dimensions when $\omega=-1$ or to 3 dimensions when $\omega=+1$. We prefer not to pick bases for these solution spaces at this point, however.

We next need certain coefficients of the 4-strand Jones-Wenzl idempotent.

\begin{lem}
\label{lem:jw}
\begin{align}
(f^{(4)})_{0123P,0123P} & = 1 \label{eq:f0123P} \\
(f^{(4)})_{0123Q,0123Q} & = 1 \label{eq:f0123Q} \\
(f^{(4)})_{2323P,2323P} & = \frac{q^{2} - 1 + q^{-2}}{q^{2} +2 + q^{-2}}  \label{eq:f2323P} \\
(f^{(4)})_{2323Q,2323Q} & =  \frac{q^{2} - 1 + q^{-2}}{q^{2} +2 + q^{-2}}  \label{eq:f2323Q} \\
(f^{(4)})_{2323P,23P3P} & = \frac {1} {\sqrt {\dim(2)\dim(P)}}\frac {1} {1+r}\frac {(q^2 + q^6) - r (1 + q^4 + q^8)} {(1 + q^4)^2}  \label{eq:f2323P,23P3P} \\
(f^{(4)})_{23232,23P32} & = -\sqrt{\frac {\dim (P)} {\dim (2)}}\frac {q^8} {(1 + q^4)^2 (1 + q^2 + 
      q^4)^2} \label{eq:f23232,23P32} \\
(f^{(4)})_{P323Q,P323Q} & = \frac{1-q^2+q^4-q^6+q^8}{\left(1+q^4\right)^2} \label{eq:fP323Q} 
\end{align}
\end{lem}
\begin{proof}
\newcommand{\TL}[1]{\input{diagrams/tikz/TL#1.tex}}
We'll just illustrate the method of calculation for $(f^{(4)})_{2323P,2323P}$.
We only need to calculate the contribution to $f^{(4)}$ of the Temperley-Lieb diagrams
$$
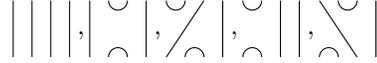
\begin{tikzpicture}[baseline=-2,scale=0.25]
        \draw (1,-1.5) -- (1,1.5);
        \draw (2,-1.5) -- (2,1.5);
        \draw (3,-1.5) -- (3,1.5);
        \draw (4,-1.5) -- (4,1.5);
\end{tikzpicture}\;, \begin{tikzpicture}[baseline=-2,scale=0.25]
        \draw (1,-1.5) -- (1,1.5);
        \draw (2,-1.5) arc (180:0:0.5);
        \draw (2,1.5) arc (-180:0:0.5);
        \draw (4,-1.5) -- (4,1.5);
\end{tikzpicture}\;, \begin{tikzpicture}[baseline=-2,scale=0.25]
        \draw (1,1.5) arc (-180:0:0.5);
        \draw (1,-1.5) -- (3,1.5);
        \draw (2,-1.5) arc (180:0:0.5);
        \draw (4,-1.5) -- (4,1.5);
\end{tikzpicture}\;, \begin{tikzpicture}[baseline=-2,scale=0.25]
        \draw (1,-1.5) arc (180:0:0.5);
        \draw (1,1.5) arc (-180:0:0.5);
        \draw (3,-1.5) -- (3,1.5);
        \draw (4,-1.5) -- (4,1.5);
\end{tikzpicture}\;, \begin{tikzpicture}[baseline=-2,scale=0.25]
        \draw (1,-1.5) arc (180:0:0.5);
        \draw (3,-1.5) -- (1,1.5);
        \draw (2,1.5) arc (-180:0:0.5);
        \draw (4,-1.5) -- (4,1.5);
\end{tikzpicture}\;$$
because when the boundary is labelled with $2323P$ along both top and bottom, any other diagram connects unequal boundary labels.
Using the formulas from \cite{morrison} (easily derivable from the earlier work of \cite{MR1446615}), we then have
\begin{align*}
f^{(4)} & = \begin{tikzpicture}[baseline=-2,scale=0.25]
        \draw (1,-1.5) -- (1,1.5);
        \draw (2,-1.5) -- (2,1.5);
        \draw (3,-1.5) -- (3,1.5);
        \draw (4,-1.5) -- (4,1.5);
\end{tikzpicture} - \frac{[2]^2}{[4]} \begin{tikzpicture}[baseline=-2,scale=0.25]
        \draw (1,-1.5) -- (1,1.5);
        \draw (2,-1.5) arc (180:0:0.5);
        \draw (2,1.5) arc (-180:0:0.5);
        \draw (4,-1.5) -- (4,1.5);
\end{tikzpicture} + \frac{[2]}{[4]} \begin{tikzpicture}[baseline=-2,scale=0.25]
        \draw (1,1.5) arc (-180:0:0.5);
        \draw (1,-1.5) -- (3,1.5);
        \draw (2,-1.5) arc (180:0:0.5);
        \draw (4,-1.5) -- (4,1.5);
\end{tikzpicture} - \frac{[3]}{[4]} \begin{tikzpicture}[baseline=-2,scale=0.25]
        \draw (1,-1.5) arc (180:0:0.5);
        \draw (1,1.5) arc (-180:0:0.5);
        \draw (3,-1.5) -- (3,1.5);
        \draw (4,-1.5) -- (4,1.5);
\end{tikzpicture} + \frac{[2]}{[4]} \begin{tikzpicture}[baseline=-2,scale=0.25]
        \draw (1,-1.5) arc (180:0:0.5);
        \draw (3,-1.5) -- (1,1.5);
        \draw (2,1.5) arc (-180:0:0.5);
        \draw (4,-1.5) -- (4,1.5);
\end{tikzpicture} + \cdots
\end{align*}
and
\begin{align*}
(f^{(4)})_{2323P,2323P} & = 1 - \frac{[2]^2}{[4]} \frac{\dim(2)}{\dim(3)} + \frac{[2]}{[4]} - \frac{[3]}{[4]} \frac{\dim(3)}{\dim(2)} + \frac{[2]}{[4]} \\
	& = \frac{q^{2} - 1 + q^{-2}}{q^{2} +2 + q^{-2}}. \qedhere
\end{align*}
\end{proof}

\begin{lem}
\label{lem:0123P}
$$S_{0123P,0123P} = \kappa = \text{$1$ or $-r$}$$
\end{lem}
\begin{proof}
Just solve Equation \eqref{eq:S2} in $\cM_{0,P}$, where we only have the $0123P,0123P$ component. Equation \eqref{eq:f0123P} tells us the coefficient of the Jones-Wenzl idempotent.
\end{proof}

When $r=1$ Equation \eqref{eq:S2} has a symmetry $S \mapsto -S$, so we may assume $\kappa = 1$ there.

It is now necessary to break into cases, to handle the two possible rotational eigenvalues.

%
%

\section{When the rotational eigenvalue is $\omega=-1$ and $r=1$.}
\label{sec:omega=-1}

We now have a large number of quadratic equations, coming from Equation \eqref{eq:S2} in 24 variables, along with further linear equations which we know can cut those 24 dimensions down to 4, and throughout we are working over $\mathbb{Q}(\sqrt{\mathrm{FPdims}})$. (This field is rather awkward; by using the lopsided conventions described in \cite{1205.2742}, we could arrange to work over $\mathbb{Q}(q)$, but this turns out to give little advantage.)

We begin by asking if we are lucky: is it possible to write any of these quadratics as a function of a single variable, using the linear equations?

It turns out that many of the components of Equation \eqref{eq:S2} can be rewritten modulo Equations \eqref{eq:lws-1} through \eqref{eq:lws-3} as quadratics in a single variable, and in particular we can find univariate equations in any of $S_{2323P,23P3P}$, $S_{2323P,23Q3P}$, $S_{2323Q,23P3Q}$, $S_{2323Q,23Q3Q}$, $S_{23P32,23P32}$, $S_{23P32,23Q32}$ or $S_{23Q32,23Q32}$. These variables only span a 2 dimensional subspace modulo Equation \eqref{eq:lws-3}, so we choose two components to analyze. Choosing components corresponding to collapsed loops reduces the amount of work required to rewrite in terms of a single variable.

\begin{lem}
\label{lem:stage11}
The components $23232,23P32$, $2323P,2323P$ of Equation \eqref{eq:S2} can be simplified using Equations \eqref{eq:lws-0} through \eqref{eq:lws-3} and Lemma \ref{lem:jw} to give
\begin{align}
S_{23P32,23P32} & = -\frac{q^4}{1 + q^2 + 2q^4 + q^6 + q^8} \notag \\
S_{2323P,23P3P}^2 & = -\frac{\left(q^{11}+q\right)^2}{2 \left(q^4+1\right) \left(q^4+q^2+1\right)^2 \left(q^{12}+q^8+q^6+q^4+1\right)}. \label{eq:stage11b}
\end{align}
\end{lem}
\begin{proof}
The matrix algebra $\cM_{2,2}$ is 7-by-7 with rows and columns indexed by the paths $21012, 21212, 21232, 23212, 23232, 23P32$ and $23Q32$. Thus
$$S^2_{23232,23P32} = \sum_\gamma S_{23232,\gamma} S_{\gamma,23P32}$$
and we easily see that only the $\gamma = 23P32$ and $\gamma = 23Q32$ terms contribute to this sum, since otherwise $S_{23232,\gamma}$ is confined to the initial arm and so by Equations \eqref{eq:lws-0}, \eqref{eq:lws-1} and \eqref{eq:lws-2} zero.

Thus we have the equation 
\begin{align*}
S_{23232,23P32} S_{23P32,23P32} + S_{23232,23Q32} S_{23Q32,23P32} & = (f^{(4)})_{23232,23P32} \\
    & = -\sqrt{\frac {\dim (P)} {\dim (2)}}\frac {q^8} {(1 + q^4)^2 (1 + q^2 +       q^4)^2} 
\end{align*}
where we've used Equation \eqref{eq:f23232,23P32}. It's now easy to see the strategy. First use Equation \eqref{eq:lws-3} to replace $S_{23232,23Q32}$ and $S_{23Q32,23P32}$ with a linear combination of coefficients which avoid $Q$. Then  modulo rotations and collapsing, we can write $S_{23232,23P32}$ (indeed, any collapsed paths that enter $P$ or $Q$ exactly once) as a known multiple of $S_{0123P,0123P}$, which we know is equal to $1$ from Lemma \ref{lem:0123P}. After these steps, the left hand side is linear in $S_{23P32,23P32}$, and collecting terms gives the desired result.
\end{proof}

The proofs of the remaining lemmas in the paper are all analogous to the above, and we omit them for brevity.

\begin{lem}
The component $2323P,23P3P$ of  Equation \eqref{eq:S2} can be simplified using Equations \eqref{eq:lws-0} through \eqref{eq:lws-3}, Lemma \ref{lem:jw} and either solution of the equations in Lemma \ref{lem:stage11} to give
\begin{align*}
S_{23P3P,23P3P} & = \frac{q^{18}+3 q^{14}-2 q^{12}+3 q^{10}-2 q^8+3 q^6+q^2}{2 \left(q^8+q^6+q^4+q^2+1\right) \left(q^{12}+2 q^8+2 q^4+1\right)}
\end{align*}
\end{lem}

\begin{lem}
The $P323Q,P323Q$ component of Equation \eqref{eq:S2} can be simplified (using everything above, with either solution of Equation \eqref{eq:stage11b}) to give
\begin{equation}
\label{eq:q-identity1}
-\frac{\left(q^8{-}q^6{-}q^4{-}q^2{+}1\right) \left(q^8{-}q^6{+}q^4{-}q^2{+}1\right) \left(q^8{+}q^6{+}3 q^4{+}q^2{+}1\right)}{\left(q^2{-}q{+}1\right) \left(q^2{+}q{+}1\right) \left(q^4{+}1\right) \left(q^4{-}q^2{+}1\right)^2 \left(q^4{-}q^3{+}q^2{-}q{+}1\right) \left(q^4{+}q^3{+}q^2{+}q{+}1\right)} = 0
\end{equation}
\end{lem}

\begin{lem}
All of the factors in Equation \eqref{eq:q-identity1} are strictly positive for $q>1$, except for $1 - q^2 - q^4 - q^6 + q^8$, which has 4 real roots at $q = \pm q_0^{\pm1}$ where $q_0$ is the largest real root, approximately $1.31228...$. The corresponding index value $q^2 + 2 + q^{-2}$ is $\frac{5+\sqrt{13}}{2}$, that is, the index of the Haagerup subfactor.
\end{lem}

Thus, by the classification results of \cite{MR1317352} 
the only possibility with $\omega=-1$ is that the standard invariant of our subfactor is exactly that of the Haagerup subfactor.

\section{When the rotational eigenvalue is $\omega=+1$.}
\label{sec:omega=+1}
\begin{lem}
We cannot have $\kappa = -r$ with $\omega=1$ in Lemma \ref{lem:0123P}, since otherwise the $0123Q,0123Q$ component of Equation \eqref{eq:S2} gives 
$$\frac{8 q^4 \left(q^8+q^6+q^4+q^2+1\right)^2 \left(q^8+q^6+2 q^4+q^2+1\right)}{\left(q^2+1\right)^8 \left(q^4-q^2+1\right)^4} = 0,$$ which is only possible if $q$ is a root of unity.
\end{lem}

\begin{lem}
\label{lem:stage21}
If $\kappa = 1$ in Lemma \ref{lem:0123P}, then the $23232,23P32$ and $2323P,2323P$ components of Equation \eqref{eq:S2} give
\begin{align*}
S_{23P32,23P32} & = 0 \\
\intertext{and}
 S_{2323P,23P3P} & = \frac{q \left(q^8+q^6+q^2+1\right)}{\sqrt{2} \sqrt{q^4+1} \left(q^4+q^2+1\right)^2} \\
 \intertext{or}
 & = \frac{q \sqrt{q^4+1}}{\sqrt{2} \left(q^4+q^2+1\right)}
\end{align*}
\end{lem}

\begin{lem}
The first case in Lemma \ref{lem:stage21} is impossible, because then the $2323Q,2323Q$ component of Equation \eqref{eq:S2} gives 
$$
\frac{8 q^6 \left(q^4-q^3+q^2-q+1\right)^2 \left(q^4+q^3+q^2+q+1\right)^2}{(q-1)^2 (q+1)^2 \left(q^2+1\right)^4 \left(q^2-q+1\right)^2 \left(q^2+q+1\right)^2 \left(q^4-q^2+1\right)^2} = 0
$$
which can not hold when $q > 1$.
\end{lem}

\begin{lem}
In the second case of Lemma \ref{lem:stage21} the $2323P,23P3P$ component of Equation \eqref{eq:S2} gives 
\begin{equation*}
S_{23P3P,23P3P} = -\frac{q^2}{2 \left(q^4+q^2+1\right)}
\end{equation*}
\end{lem}

\begin{lem}
Finally, in the second case of Lemma \ref{lem:stage21}, using the conclusion of the previous lemma we find that the $P323Q,P323Q$ component of Equation \eqref{eq:S2} gives
\begin{equation*}
\frac{-q^{20}+2 q^{10}-1}{\left(q^2+1\right)^4 \left(q^4-q^2+1\right)^3} = 0
\end{equation*}
which is impossible when $q > 1$.
\end{lem}

Thus we see that the rotational eigenvalue $\omega=+1$ was impossible, and we have completed the proof of the theorem.

\newcommand{\urlprefix}{}

\bibliographystyle{alpha}
\bibliography{../../bibliography/bibliography}

This paper is available online at 
\arxiv{1302.5148}, under a ``Creative Commons-By Attribution'' license.
It has been accepted for publication in the \emph{Bulletin of the London Mathematical Society} as of 2013/10/14.
MSC classes:	46L37 (Primary), 18D05, 57M20 (Secondary)

\end{document}